\documentclass[12pt,reqno]{amsart}

\usepackage{amssymb,latexsym}

\usepackage{enumerate}

 \usepackage[french,english]{babel}
 \usepackage{amsmath}
\usepackage{graphicx}
\usepackage{amssymb}
\usepackage{bbm}
\usepackage{amsthm,mathtools}
\usepackage{ulem}
\usepackage{geometry}
\usepackage{tikz-cd}
\usepackage{hyperref}
\usepackage{mathrsfs}
\usepackage[colorinlistoftodos]{todonotes}
\usepackage{enumitem}
\usepackage{verbatim}
\usepackage[foot]{amsaddr}

\allowdisplaybreaks
\makeatletter

\@namedef{subjclassname@2010}{

  \textup{2020} Mathematics Subject Classification}

\makeatother
\newtheorem{thm}{Theorem}[section]
\newtheorem*{thm*}{Theorem}

\newtheorem{lem}[thm]{Lemma}

\theoremstyle{definition}

\numberwithin{equation}{section}

\newcommand{\M}{{\mathcal M}}

\newcommand{\F}{\mathcal{F}}

\usepackage{hyperref}
\hypersetup{hypertex=true,colorlinks=true,linkcolor=blue,anchorcolor=blue,citecolor=blue}
\newcommand{\newabstract}[1]{%
  \par\bigskip
  \csname otherlanguage*\endcsname{#1}%
  \csname captions#1\endcsname
  \item[\hskip\labelsep\scshape\abstractname.]
}

\begin{document}

\baselineskip=17pt

\title[Large values of quadratic character sums revisited]{Large values of quadratic character sums revisited}

\author{Zikang Dong}
\author{Ruihua Wang}
\author{Weijia Wang}
\author{Hao Zhang}
\address[Zikang Dong]{School of Mathematical Sciences, Tongji University, Shanghai 200092, P. R. China}
\address[Ruihua Wang]{School of Fundamental Sciences, Hainan Bielefeld University of Applied Sciences, Danzhou 578101, P. R. China}
\address[Weijia Wang]{Morningside Center of Mathematics, Academy of Mathematics and Systems Science, Chinese Academy of Sciences, Beijing 100190, P. R. China}
\address[Hao Zhang]{School of Mathematics, Hunan University, Changsha 410082, P. R. China}
\email{zikangdong@gmail.com}
\email{ruih.wan9@gmail.com}
\email{weijiawang@amss.ac.cn}
\email{zhanghaomath@hnu.edu.cn}

\date{\today}

\begin{abstract} 
We study large values of quadratic character sums with summation lengths exceeding the square root of the modulus. 
Assuming the Generalized Riemann Hypothesis, we obtain a new Omega result.

\end{abstract}


\maketitle

\section{Introduction}
The study of character sums has a long history in analytic number theory. In 1918,  P\'{o}lya and  Vinogradov proved $\sum_{n\leq x}\chi(n)\leq \sqrt{q}\log q$ uniformly for any non-trivial Dirichlet character $\chi ({\rm mod}\; q)$. Later in 1977,  Montgomery and
 Vaughan improved this upper bound to $\ll \sqrt{q}\log_2q$ assuming the Generalized Riemann Hypothesis (GRH) in \cite{MV77}. Here $\log_j(x)$ denotes the $j$-th iterated logarithmic function. The latter bound is optimal up to the implied constant as a uniform bound. More precisely,  Paley showed in \cite{Paley} that there exist  quadratic characters with sums that are $\gg \sqrt{q}\log_2q$. Note that this kind of Omega results are often  attained when $x$ is relatively large, with size $\asymp q$. If $x$ is smaller, only weaker results can be get.

For any $x\le q$, let $$\Delta_q(x):=\max_{\chi\neq\chi_0({\rm mod}\,q)}\Bigg|\sum_{n\le x}\chi(n)\Big|.$$
In their seminal paper \cite{GS01}, A. Granville and K. Soundararajan proved remarkable results using  probabilistic methods. They transfer the character sums to sums of Steinhaus random multiplicative functions and then study the asymptotic behaviors of character sums by analyzing the moments of sums of Steinhaus random multiplicative functions. Their results are divided into different ranges of $x$.  Most of them have been improved subsequently by Munsch \cite{Munsch}, Hough \cite{Hough} and La Bret\`eche and Tenenbaum \cite{BT}. 

 In \cite{Munsch}, Munsch showed  when $\log q\leq x\leq \exp(\sqrt{\log q})$,
\begin{equation*}
    \Delta_q(x)\ge \Psi\bigg(x,\big(\tfrac14+o(1)\big)\frac{\log q\log_2q}{\max\{\log_2x-\log_3q,\log_3q\}}\bigg).
\end{equation*}
 Hough proved in \cite{Hough} that when $\exp(4\sqrt{\log q\log_2q}\log_3q)\le x\le\exp((\log q)^{\frac12+\delta})$, we have 
\[
\Delta_q(x)\ge \sqrt{x}\exp\bigg((1+o(1))\sqrt{\frac{\log q}{\log_2 q}}\bigg).
\]
 La Bret\`eche and  Tenenbaum showed in \cite{BT} that when $\exp((\log q)^{\frac12+\delta})\le x\le q^{\frac12}$, we have 
\[\Delta_q(x)\ge \sqrt x\exp\bigg((\sqrt2+o(1))\sqrt{\frac{\log( q/x)\log_3(q/x)}{\log_2(q/x)}}\bigg).\]

These results were generalized to large zeta sums by part of the authors in \cite{DWZ}, and to quadratic character sums by Dong and Zhang \cite{DZ}. When the summation length is $>\sqrt q$,   Hough also showed a dual result for $\exp(4\sqrt{\log q\log_2q}\log_3q)\le x\le\exp((\log q)^{\frac12+\delta})$. That is, the Omega result for $\sum_{n\le q/x}\chi(n)$.  A generalization to the quadratic character sum was also made in \cite{DSWZ}. In this paper, we prove a dual result of the above result of  La Bret\`eche and  Tenenbaum  \cite{BT} for the quadratic character sums.

 Denote by $\F$ the set of all fundamental discriminants $d$, and let $X$ be large. 

\begin{thm}\label{thm1.1}

Assume GRH. Let $X$ be large and $\exp((\log X)^{\frac12+\varepsilon})<x<X^{\frac12}$. We have
$$\max_{X<|d|\le2X\atop d\in\F}\sum_{n\le |d|/x}\chi_d(n)\ge\sqrt{\frac Xx}\exp\bigg((1+o(1))\sqrt{\frac{\log (\sqrt {X}/x)\log_3(\sqrt{X}/x)}{\log_2(\sqrt X/x)}}\bigg).$$ 
\end{thm}

Our proof is based on  results in \cite{BT} and \cite{DM}. The main method  used is called the resonance method, which was highly developed by Soundararajan \cite{Sound08}.

%\begin{thm}\label{thm1.2}
 %  Assume GRH. Let $x\le\exp((\log X)^{\frac12})$. Then we have
%$$\max_{X<|d|\le2X\atop d\in\F}\sum_{n\le |d|/x}\chi_d(n)\ge\frac{\sqrt X}{x}\Psi\bigg(x,\big(\tfrac14+o(1)\big)\frac{\log X\log_2X}{\max\{\log_2x-\log_3X,\log_3X\}}\bigg).$$
%\end{thm}
%%%%%%%%%%%%%%%%%%%%%%%%%%%%%%%%%%%%%%%%%%%%%%%%%%%%%%%%%%%%%%%%%%%%%%%%%%%%%%

   \section{Preliminary Lemmas}\label{sec2}
    The following Fourier expansion for character sums was first showed by P\'olya.
 \begin{lem}\label{lem2.1}
 Let $\chi({\rm mod}\;q)$ be any primitive character and $0<\alpha<1$. Then we have
$$\sum_{n\le \alpha q}\chi(n)=\frac{\tau(\chi)}{2\pi i}\sum_{1\le|m|\le z}\frac{\overline\chi(m)}{m}(1-e(-\alpha m))+O(1+q\log q/z),$$
where $\tau(\chi):=\sum_{n\le q}\chi(n)e(n/q)$ is the Gauss sum and $e(a):=e^{2\pi ia}$.
\end{lem}
\begin{proof}
This is \cite[p.311, Eq. (9.19)]{MVbook}.
\end{proof}
The following conditional estimate for characters has a good error term in use.
    \begin{lem}\label{lem2.2}
	Assuming GRH. Let $n=n_0n_1^2$ be a positive integer with $n_0$ the square-free part of $n$.
	%Let $\beta\geq 2$ and $\frac{1}{2}<\sigma< 1$ with $\sigma>\frac{\beta -1}{\beta}$.  
	 Then for any $\varepsilon>0$, we obtain
	\begin{align*}
	\sum_{|d|\le X\atop d\in\F} \chi_{d}(n)=\frac{X}{\zeta(2)}\prod_{p|n}\frac{p}{p+1}{1}_{n=\square}+ O\left(X^{\frac12+\varepsilon}f(n_0)g(n_1)\right),
	\end{align*}
	where   ${{1}}_{n=\square}$ indicates the indicator function of the square numbers, and
    $$f(n_0)=\exp((\log n_0)^{1-\varepsilon}),\;\;\;\;g(n_1)=\sum_{d|n_1}\frac{\mu(d)^2}{d^{\frac12+\varepsilon}}.$$
\end{lem}
\begin{proof}
    This follows directly from Lemma 1 of \cite{DM}.
\end{proof}
On the one hand, it is clear that 
$$f(n_0)\le n_0^\varepsilon\le n^\varepsilon,\;\;\;\;g(n_1)\le n_1^\varepsilon\le n^\varepsilon.$$
On the other hand, if we denote the largest prime factor of $n$ by $P_+(n)$, then $n_0,n_1\le \prod_{p\le P_+(n)}p$.
So easily we have
$$f(n_0)\le\exp\big(P_+(n)^{1-\varepsilon}\big),\;\;\;\;\;g(n_1)\le\exp\big(P_+(n)^{\frac12-\varepsilon}\big).$$

\begin{lem}\label{GCD}
    Let $\M$ be any set of positive squarefree integers with $|\M|=N$. Then as $N\to\infty$, we have
    $$\max_{|\M|=N}\sum_{m,n\in\M}\sqrt{\frac{(m,n)}{[m,n]}}=N\exp\bigg((2+o(1))\sqrt{\frac{\log N\log_3N}{\log_2N}}\bigg).$$
\end{lem}
\begin{proof}
    This is Eq. (1.5) of \cite{BT}.
\end{proof}
Note that in the proof of the above lemma, the choice for the set $\M$ satisfies $y_\M:=\max_{m\in\M}P_+(m)\le (\log N)^{1+o(1)}$.

\section{Proof of Theorem \ref{thm1.1}}\label{sec4}
Choose $z=\sqrt{|d|x}\log |d|$, by Lemma \ref{lem2.1} we have
\begin{align*}&\sum_{n\le |d|/x}\chi_d(n)\\&=\frac{\tau(\chi_d)}{2\pi i}\sum_{1\le|m|\le z}\frac{\chi_d(m)}{m}\big(1-e(-m/x)\big)+O(\sqrt{|d|/x})\\
&=\frac{\tau(\chi_d)}{2\pi i}\sum_{1\le|m|\le z}\frac{\chi_d(m)}{m}\big(1-c(m/x)\big)+\frac{\tau(\chi_d)}{2\pi }\sum_{1\le|m|\le z}\frac{\chi_d(m)}{m}s(m/x)+O(\sqrt{|d|/x}),\end{align*}
where
$$e(a):=e^{2\pi ia},\;\;c(a):=\cos2\pi a,\;\;s(a):=\sin 2\pi a.$$
Let $$C_d(z):=\sum_{|m|\le z}\frac{\chi_d(m)}{m}\big(1-c(m/x)\big),$$
and
$$S_d(z):=\sum_{|m|\le z}\frac{\chi_d(m)}{m}s(m/x).$$
If $\chi_d(-1)=1$, then
$C_d(z)=0$.
If $\chi_d(-1)=-1$,  then $S_d(z)=0$, and
$$\Big|\sum_{n\le |d|/x}\chi_d(n)\Big|=\frac{\sqrt {|d|}}{2\pi}|C_d(z)|+O(\sqrt{|d|/x}).$$
Thus we have 
\begin{align}\label{polya}
\max_{X<|d|\le 2X\atop d\in\F}\Big|\sum_{n\le |d|/x}\chi_d(n)\Big|\ge\max_{X<|d|\le 2X\atop d\in\F}\frac{\sqrt {|d|}}{2\pi}|C_d(z)|+O(\sqrt{|d|/x}).\end{align}

  Let $N=\lfloor X^{\frac12-\delta}/x\rfloor$, and $\M$ be a set such that Lemma \ref{GCD} holds.  We define the resonator $$R_d:=\sum_{m\in \M}\chi_d(n),$$and 
$$M_1(R,X):=\sum_{X<|d|\le 2X\atop d\in\F}R_d^2,$$
$$M_2(R,X):=\sum_{X<|d|\le 2X\atop d\in\F}R_d^2C_d(z)^2.$$
Then
\begin{align*}\max_{X<|d|\le 2X\atop d\in\F}C_d(z)^2\ge\frac{M_2(R,X)}{M_1(R,X)}.\end{align*}
For $M_1(R,X)$, it holds that 
\begin{align*}M_1(R,X)&=\frac{X}{\zeta(2)}\sum_{m,n\in\M\atop mn=\square}\prod_{p|mn}\frac{p}{p+1}+O\Big(X^{\frac12+\varepsilon}\sum_{m,n\in\M}1\Big)\\
&=\frac{X}{\zeta(2)}\sum_{m\in\M}\prod_{p|m}\frac{p}{p+1}+O\big(X^{\frac12+\varepsilon}N^2\big)\\
&\le\frac{X}{\zeta(2)}N+O\big(X^{\frac12+\varepsilon}N^2\big)\\
&\ll{X}N.\end{align*}
For $M_2(R,X)$, we have
\begin{align*}
&M_2(R,X)\\&=\frac{X}{\zeta(2)}\sum_{m,n\in\M}\sum_{1\le|k|,|\ell|\le z\atop mnk\ell=\square}\frac{(1-c(k/x))(1-c(\ell/x))}{k\ell}\prod_{p|mnk\ell}\frac{p}{p+1}+O\Big(X^{\frac12+\varepsilon}\sum_{m,n\in\M}\sum_{k,\ell\le x}1\Big)\nonumber\\
&=\frac{2X}{\zeta(2)}\sum_{m,n\in\M}\sum_{k,\ell\le z\atop mnk\ell=\square}\frac{4s(k/2x)^2s(\ell/2x)^2}{k\ell}\prod_{p|mnk\ell}\frac{p}{p+1}+O\big(X^{\frac12+\varepsilon}N^2x^2\big)\\
&\ge\frac{8X}{\zeta(2)}\sum_{m,n\in\M}\sum_{k,\ell\le x/2\atop mnk\ell=\square}\frac{s(k/2x)^2s(\ell/2x)^2}{k\ell}\prod_{p|mnk\ell}\frac{p}{p+1}+O\big(X^{\frac12+\varepsilon}N^2x^2\big)\\
&\ge\frac{8X}{\zeta(2)}\Big(\frac2\pi\Big)^4\sum_{m,n\in\M}\sum_{k,\ell\le x/2\atop mnk\ell=\square}\frac{(2\pi k/2x)^2(2\pi \ell/2x)^2}{k\ell}\prod_{p|mnk\ell}\frac{p}{p+1}+O\big(X^{\frac12+\varepsilon}N^2x^2\big)\\
&\gg\frac{X}{x^4}\sum_{m,n\in\M}\sum_{k,\ell\le x/2\atop mnk\ell=\square}{k\ell}\prod_{p|mnk\ell}\frac{p}{p+1}+O\big(X^{\frac12+\varepsilon}N^2x^2\big).
\end{align*}
So
\begin{equation}\frac{M_2(R,X)}{M_1(R,X)}\gg\frac{1}{Nx^4}\sum_{m,n\in\M}\sum_{k,\ell\le x/2\atop mnk\ell=\square}k\ell\prod_{p|mnk\ell}\frac{p}{p+1}+O(X^{-\delta+\varepsilon}x).\label{mainterm}\end{equation}
Let
$$I_2(R,X):=\sum_{m,n\in\M}\sum_{k,\ell\le x/2\atop mnk\ell=\square}k\ell\prod_{p|mnk\ell}\frac{p}{p+1}.$$
 We have 
$$I_2(R,X)
\ge\prod_{p\le X}\frac{p}{p+1}\sum_{m,n\in\M}\sum_{k,l\le x/2 \atop mk=n\ell}k\ell
\ge(\log X)^{-c}\sum_{m,n\in\M}\sum_{k,l\le x/2 \atop mk=n\ell}k\ell,$$
for some $c>0$.
For fixed $m,n$, $mk=n\ell$ implies $k=nL/(m,n)$ and $\ell=mL/(m,n)$ for some integer $L$. Since $\max\M\le2\min\M$, we have for the inner sum
\begin{align*}\sum_{k,\ell\le x/2\atop mk=n\ell}k\ell&=\sum_{L\le\max\{\frac{m}{(m,n)},\frac{n}{(m,n)}\}}\frac{mn}{(m,n)^2}L^2\\&\gg\frac{[m,n]}{(m,n)}\bigg(\frac{x/2}{\max\{\frac{m}{(m,n)},\frac{n}{(m,n)}\}}\bigg)^3\\&\ge\frac{[m,n]}{(m,n)}\bigg(\frac{x/2}{\sqrt{2\frac{m}{(m,n)}\frac{n}{(m,n)}}}\bigg)^3\\&\gg x^3\sqrt{\frac{(m,n)}{[m,n]}}
.\end{align*}
It follows that
$$I_2(R,X)\gg x^3(\log X)^{-c}\sum_{m,n\in\M\atop[m,n]/(m,n)\le x^2/8}\sqrt{\frac{(m,n)}{[m,n]}}.$$
We have by Lemma \ref{GCD}
\begin{align*}
&\sum_{m,n\in\M\atop[m,n]/(m,n)\le x^2/8}\sqrt{\frac{(m,n)}{[m,n]}}\\
&=\bigg(\sum_{m,n\in\M
}\sqrt{\frac{(m,n)}{[m,n]}}-\sum_{m,n\in\M
\atop [m,n]/(m,n)> x^2/8}\sqrt{\frac{(m,n)}{[m,n]}}\bigg)\\
&\gg|\M|\exp\bigg((2+o(1))\sqrt{\frac{\log(X^{\frac12-\delta}/x)\log_3(X^{\frac12-\delta}/x)}{\log_2(X^{\frac12-\delta}/x)}}\bigg).
\end{align*}
Here in the last step we used
\begin{align*}\sum_{m,n\in\M
\atop [m,n]/(m,n)> x^2/8}\sqrt{\frac{(m,n)}{[m,n]}}&\ll x^{-2\eta}\sum_{m,n\in\M
}\bigg({\frac{(m,n)}{[m,n]}}\bigg)^{\frac12-\eta}\\
&\ll x^{-2\eta}\prod_{p\le y_\M}\bigg(1+\frac{2}{p^{\frac12-\eta}-1}\bigg)\\
&\ll x^{-2\eta}\exp\big( y_\M^{\frac12+\eta}\big)\\
&\ll x^{-2\eta}\exp\big( (\log (X^{\frac12-\delta}/x))^{\frac12+\eta+o(1)}\big)\\
&\ll \exp\big(-\tfrac23\delta(\log X)^{\frac12+\delta}\big)\exp\big((\log X)^{\frac12+\frac{2}{3}\delta}\big)\\
&\ll1.
\end{align*}
with $\eta=\delta/3$, $y_\M=\max_{m\in\M} P_+(m)\le(\log (X^{\frac12-\delta}/x))^{1+o(1)}$ and $x>\exp((\log X)^{\frac12+\delta}).$
Inserting into \eqref{mainterm}, we have 
\begin{align*}
\max_{X<|d|\le 2X\atop d\in\F}C_d(z)^2&\gg \frac x {N}(\log X)^{-c}\sum_{m,n\in\M\atop[m,n]/(m,n)\le x^2/2}\sqrt{\frac{(m,n)}{[m,n]}}+O(X^{-\delta+\varepsilon}x)\\
&\gg  x{(\log X)^{-c}}\exp\bigg((2+o(1))\sqrt{\frac{\log N\log_3N}{\log_2N}}\bigg)\\
&\ge x\exp\bigg((2+o(1))\sqrt{\frac{\log (X^{\frac12-\delta}/x)\log_3(X^{\frac12-\delta}/x)}{\log_2(X^{\frac12-\delta}/x)}}\bigg),
\end{align*}
where we have used Lemma \ref{GCD}.

\section*{Acknowledgements}
Z. Dong is supported by the Shanghai Magnolia Talent Plan Pujiang Project (Grant No. 24PJD140) and the National
	Natural Science Foundation of China (Grant No. 	1240011770). W. Wang is supported by the National
	Natural Science Foundation of China (Grant No. 1250012812). H. Zhang is supported by the Fundamental Research Funds for the Central Universities (Grant No. 531118010622), the National
	Natural Science Foundation of China (Grant No. 1240011979) and the Hunan Provincial Natural Science Foundation of China (Grant No. 2024JJ6120).

\normalem

\end{document}